\documentclass[12pt]{article}
\usepackage{amsthm,amsmath,amsfonts}
\newtheorem{theorem}{Theorem}
\newtheorem*{thm*}{Theorem}
\newtheorem{lemma}{Lemma}
\begin{document}
\title{ Some Generalizations of the Hellinger Theorem for Second Order Difference Equations
with Matrix Elements}
\date{}
\author{
A.S. Osipov \thanks{Scientific-Research Institute for System
Studies, Russian Academy of Sciences, e-mail:
\it{osipa68@yahoo.com}}
\thanks{This work was supported by RFBR: project 11-01-00790}}
\maketitle
\begin{abstract}
We obtain several generalizations the Hellinger theorem about $l^2$
solutions of difference equations:  instead of second order
equations  and $ l^2$-solutions, we consider second-order equations
with matrix coefficients and their solutions in $l^p,\; 1\le p\le
\infty$. In particular it is shown that for a certain class of
symmetric difference equations an  analog of this theorem holds
 for $1\le p \le 2$, but it does not hold  for  $p>2$.
\end{abstract}

\section{Introduction}
 In the study of the spectral properties of infinite Jacoby matrices and
analytic properties of  continued $J$-fractions,  a significant
place belongs to the result established  by E. Hellinger ~\cite{H1,
H-W, W}:
\begin{thm*}
Suppose that  for some $z=z_0\in \mathbb{C}$,  any solution
$u=u(z)=(u_{i}(z))_{i=0}^\infty$ of the
infinite system of the difference equations
\begin{eqnarray*}
a_{i-1}u_{i-1}+b_i u_i +a_i u_{i+1}=z u_i,\quad
i\ge 1,\\
a_i,\;b_i \in \mathbb{C},\quad
a_i\ne 0,
\end{eqnarray*}
satisfies the condition $\sum_{i=0}^\infty |u(z_0)|^2<\infty$
(and therefore,  belongs to the space $l^2$).  Then, for any
$z\in \mathbb{C}$ and $M>0$,  the series $\sum_{i=0}^\infty |u(z)|^2$
converges uniformly for $|z-z_0|<M$.
\end{thm*}

This  theorem was applied to the study of the essential spectrum
of  second order difference operators in ~\cite{Be}.  It should be
observed that since  $z_0$ is an {\it arbitrary} complex number,
the Hellinger theorem makes the study of the deficiency indices of
symmetric  second order difference operators more simple than  a
similar investigation for differential Sturm-Liouville operators.
For some classes of difference operators of an arbitrary order, an
extension of the Hellinger theorem was obtained in ~\cite{Os1} (
with the space $l^2$ replaced by $l^p,\;1\le p \le \infty$).
 The goal of this paper is to find a similar extension for the
 second order difference equations with matrix coefficients.These
 issues are essential in the analysis of properties of continued
 fractions with matrix (or operator) elements ~\cite{S-I,Os11,I}.

\section{Preliminaries}
As mentioned above, there is a connection between the Hellinger
theorem and some spectral properties of the infinite Jacoby
matrices. Instead of a Jacoby matrix, here we consider the infinite
three-diagonal matrix  $A=(A_{i,j})_{i,j=0}^\infty$
$$
A=\begin{pmatrix} A_{0,0}&A_{0,1}&O&O&\dots\\
   A_{1,0}&A_{1,1}&A_{1,2}&O&\dots\\
      O&A_{2,1}&A_{2,2}&A_{2,3}&\dots\\
\vdots&\vdots&\vdots&\vdots&\ddots\\
\end{pmatrix},
$$
where $A_{i,j}$ is a square matrix of order $n$ whose elements are
complex numbers; $O$ is a zero matrix of order $n$. Also assume that
$A_{i+1,i},A_{i,i+1},\; i \ge 0 $ are invertible. Then $A$ generates
a linear operator in the space $l_n^2$ of sequences
$u=(u_0,u_1,\dots)$, where the vector column $u_j \in \Bbb C^n$ with
inner product $ (u,v)=\sum_{j=0}^{\infty} v_j^*u_j$. For this
operator we will keep the same notation $A$.

To the matrix $A$ we assign second-order finite-difference
equations in the matrices $Y_j, Y_j^{+}$ of order $n$:
\begin{eqnarray}
\label{fir}
l(Y)_j\equiv A_{j,j-1}Y_{j-1}+A_{j,j}Y_{j}+A_{j,j+1}Y_{j+1}=z Y_{j},\\
\label{sec} l^{+}(Y^+)\equiv
Y_{j-1}^{+}A_{j-1,j}+Y_{j}^{+}A_{j,j}+Y_{j+1}^{+}A_{j+1,j}= z
Y_{j}^{+}, \\
j \ge 0, \;  z \in \Bbb C; \quad A_{0,-1}=A_{-1,0}=-E,\nonumber
\end{eqnarray}
where $E$ is a unit matrix.

Denote by
$P(z)=\{P_j(z)\}_{j=-1}^\infty,\,Q(z)=\{Q_j(z)\}_{j=-1}^\infty,\;
P^{+}(z)=\{P^{+}_j(z)\}_{j=-1}^\infty,$ \linebreak
$Q^{+}(z)=\{Q_j^{+}(z)\}_{j=-1}^\infty $ the solutions of
(\ref{fir}) and (\ref{sec}) respectively, satisfying the initial
conditions
$$
\gathered
P_{-1}(z)=P_{-1}^{+}(z)=Q_0(z)=Q_0^{+}(z)=E;\\
P_{0}(z)=P_{0}^{+}(z)=Q_{-1}(z)=Q_{-1}^{+}(z)=O;
\endgathered
$$
These solutions are matrix polynomials in $z$; $Q(z), Q^{+}(z)$ and
$P(z), P^{+}(z)$ are analogs of the polynomials of the first and the
second kind for scalar Jacobi matrices ~\cite{Ah}. They play an
important role in the spectral analysis of the corresponding
operator $A$. For example, if $A$ is bounded, then for any $z$ from
its resolvent set
$$ \varlimsup_{m\to \infty} \Vert Q_m(z) \Vert ^{\frac{1}{m}}>1,
\quad \varlimsup_{m\to \infty} \Vert Q_m^+(z) \Vert
^{\frac{1}{m}}>1,
$$
see ~\cite{Os2} for more details. The following identities, which
can be verified by induction on $j \ge 0$, are valid ~\cite{Os2}:

\begin{equation}
\label{ind1} P_jQ_j^{+}-Q_jP_j^+=O;\;
P_{j+1}Q_{j}^{+}-Q_{j+1}P_{j}^+=A_{j,j+1}^{-1};\;
Q_{j}P_{j+1}^{+}-P_{j}Q_{j+1}^{+}=A_{j+1,j}^{-1};
\end{equation}
\begin{eqnarray}
\label{ind2} Q_{j+1}^{+}A_{j+1,j}Q_{j}-Q_j^{+}A_{j,j+1}Q_{j+1}=
P_{j+1}^{+}A_{j+1,j}P_{j}-P_j^{+}A_{j,j+1}P_{j+1}=O; \nonumber \\
P_{j+1}^{+}A_{j+1,j}Q_{j}-P_j^{+}A_{j,j+1}Q_{j+1}=
Q_j^{+}A_{j,j+1}P_{j+1}-Q_{j+1}^{+}A_{j+1,j}P_{j}=E,
\end{eqnarray}
(the parameter $z$ is omitted for convenience of notation).

Now let $F_j, \; j \ge 0$ be an arbitrary sequence of $n \times n$
matrices. Assuming that $P(z), Q(z)$ and $P^{+}(z), Q^{+}(z)$ are
known, we solve the inhomogeneous equations
\begin{eqnarray}
\label{inh1}
l(U)_j-zU_j=F_j, \\
\label{inh2}
l^{+}(U^+)_j-zU_j^+=F_j, \\
\nonumber j \ge 0, z \in \Bbb C
\end{eqnarray}
by variation of constants on setting
\begin{eqnarray}
\label{u1}
U_j(z)=Q_j(z)C_j^1+P_j(z)C_j^2,\\
\label{u2}
 U_j^{+}(z)=C_j^{1,+}Q_j^{+}(z)+C_j^{2,+}P_j^{+}(z).
\end{eqnarray}
\begin{lemma}
For the matrices $C_j^1,\,C_j^2,\,C_j^{1,+},\,C_j^{2,+}$ the
following recursive representations are valid:
\begin{eqnarray}
\label{c1}
C_j^1=C_k^1-\sum_{i=k}^{j-1}P_i^{+}(z)F_i, \\
\label{c2}
C_j^2=C_k^2+\sum_{i=k}^{j-1}Q_i^{+}(z)F_i;\\
\label{cp1}
C_j^{1,+}=C_k^{1,+}-\sum_{i=k}^{j-1}F_iP_i(z),\\
\label{cp2}
C_j^{2,+}=C_k^{2,+}+\sum_{i=k}^{j-1}F_iQ_i(z);\\
\nonumber k=0,1,\dots \qquad j=k+1,k+2,\dots
\end{eqnarray}
Also, for $j=-1,0$
\begin{equation*}
U_j(z)=Q_j(z)C_0^1+P_j(z)C_0^2; \quad
U_j^{+}(z)=C_0^{1,+}Q_j^{+}(z)+C_0^{2,+}P_j^{+}(z),
\end{equation*}
where $C_0^1,C_0^2,C_0^{1,+},C_0^{2,+}$ - arbitrary constant
matrices.
\end{lemma}
\begin{proof}
For $j \ge 0$ consider the system
\begin{eqnarray*}
\begin{cases}
Q_j(z)\Delta C_{j+1}^1+P_j(z)\Delta C_{j+1}^2 = 0,\\
A_{j,j+1}(Q_{j+1}(z)\Delta C_{j+1}^1+P_{j+1}(z)\Delta C_{j+1}^2) =
F_j
\end{cases}
\end{eqnarray*}
where $\Delta C_{j+1}^i=C_{j+1}^i-C_j^i, \, i=1,2.$ One can check by
direct substitution that if the matrices $C_j^i$ are chosen in this
way, then the sequence $U_j(z)$ defined by (\ref{u1}) is a solution
of (\ref{inh1}). Let us show that this system has a unique solution.
Indeed, multiplying the first equation of the system on the left by
$P_{j+1}^{+}(z)A_{j+1,j}$ and the second equation by $-P_j^{+}(z)$
and summing the resulting equations, we obtain
\begin{eqnarray*}
(P_{j+1}^{+}(z)A_{j+1,j}Q_j(z)-P_j^{+}(z)A_{j,j+1}Q_{j+1}(z))\Delta
C_{j+1}^1+\\
+(P_{j+1}^{+}(z)A_{j+1,j}P_{j}(z)-P_j^{+}(z)A_{j,j+1}P_{j+1}(z))\Delta
C_{j+1}^2=-P_j^{+}(z)F_j
\end{eqnarray*}
Taking into account the identities (\ref{ind2}), we find that
\begin{equation}
\label{d1}
 \Delta C_{j+1}^1=-P_j^{+}(z)F_j.
\end{equation}
Using similar arguments, we can show that
\begin{equation}
\label{d2}
\Delta C_{j+1}^2=Q_j^{+}F_j.
\end{equation}
Hence the above system has a unique solution. Summing (\ref{d1}) and
(\ref{d2}) by $i=k,k+1,\dots,j-1$ we finally obtain (\ref{c1}) and
(\ref{c2}).The formulas (\ref{cp1})-(\ref{cp2}) can be obtained in a
similar manner by using the system
\begin{eqnarray*}
\begin{cases}
\Delta C_{j+1}^{1,+}Q_j^{+}(z)+\Delta C_{j+1}^{2,+}P_j^{+}(z) = 0,\\
(\Delta C_{j+1}^{1,+}Q_{j+1}^{+}(z)+\Delta
C_{j+1}^{2,+}P_{j+1}^{+}(z))A_{j+1,j} = F_j
\end{cases}
\end{eqnarray*}
with respect to $\Delta C_{j+1}^{i,+}=C_{j+1}^{i,+}-C_j^{i,+}, \,
i=1,2.$
\end{proof}
\section{Main results}
Now for $ 1 \le p \le \infty $ consider the Banach spaces $l^p_n$
of sequences $u=(u_{-1},u_0,u_1,\dots)$, such that the vector
column $u_j \in \Bbb C^n,\, $ with the norm $\Vert u \Vert_p =
(\sum_j \vert u_j\vert^p)^{1/p}<\infty,\, 1 \le p <\infty $, where
$ \vert \cdot \vert $ is a certain vector norm. For the case $
p=\infty \, \Vert u \Vert_\infty = \sup_{j}\vert u_j \vert $.

Alongside with (\ref{fir})-(\ref{sec}), consider the equations in
the vectors $u_j, v_j \in \Bbb C^n$:
\begin{eqnarray}
\label{firv}
l(u)_j\equiv& A_{j,j-1}u_{j-1}+A_{j,j}u_{j}+A_{j,j+1}u_{j+1}=z u_{j},\\
\label{secv} l^{+}(v^*)\equiv&
v_{j-1}^{*}A_{j-1,j}+v_{j}^{*}A_{j,j}+v_{j+1}^{*}A_{j+1,j}= z v_{j}^{*},\\
\nonumber  &j \ge 0, z \in \Bbb C.
\end{eqnarray}

Since the polynomials $P(z), Q(z)$ and $P^{+}(z), Q^{+}(z)$ form
the fundamental systems of solutions of the equations (\ref{fir})
and (\ref{sec}) respectively, one can easily see that if all the
solutions of (\ref{firv}) belong to the space $l^p_n$, then
$M_k^p(z) \to \infty$ as $k \to \infty,$ where
\begin{equation}
\label{mp}
 M_k^p(z) \equiv \max{\{\left(\sum_{j=k}^{\infty}\Vert
P_j(z)\Vert^p \right)^{1/p},\left(\sum_{j=k}^{\infty}\Vert
Q_j(z)\Vert^p\right)^{1/p}\}}
\end{equation}
and $\Vert \cdot \Vert$ is a matrix norm. Similarly, if all the
solutions of (\ref{secv}) belong to the space $l^q_n, \, 1 \le q <
\infty,$ then $M_k^{q,+}(z) \to \infty$ as $k \to \infty,$ where
\begin{equation}
\label{mq}
 M_k^{q,+}(z) \equiv \max{\{\left(\sum_{j=k}^{\infty}\Vert
P_j^{+}(z)\Vert^q \right)^{1/q},\left(\sum_{j=k}^{\infty}\Vert
Q_j^{+}(z)\Vert^q\right)^{1/q}\}}.
\end{equation}

\begin{theorem}
\label{th1} If all solutions of (\ref{firv}) are in $ l_n^p , 1\le
p \le \infty $ for some $z=z_0\in \mathbb{C}$ and all solutions of
the equaton (\ref{secv}) are in $ l_n^q,\, 1/q+1/p=1 $, then for
any $ z \in \mathbb{C} $ all solutions of (\ref{firv}) and
(\ref{secv}) are in $ l_n^p $ and $ l_n^q $ respectively.
\end{theorem}
\begin{proof}
First consider the case $ 1 < p < \infty$. The equation (\ref{fir})
can be written in the form
\begin{equation*}
l(Y(z))_{j}-z_{0} Y_j(z)=(z-z_{0})Y_j(z), \quad j \ge 0.
\end{equation*}
This equation can be reduced to (\ref{inh1}) by setting $F_j =
(z-z_0)Y_j$.  Substituting (\ref{c1}) and (\ref{c2}) into (\ref{u1})
we obtain the following representation for $Y_j(z)$:
\begin{eqnarray}
\label{rep1} \nonumber
Y_j(z)= Q_j(z_0)C_k^1+P_j(z_0)C_k^2+ \\
+(z-z_0)\sum_{i=k}^{j-1}\left(P_j(z_0)Q_i^{+}(z_0)-Q_j(Z_0)P_i^{+}(z_0)
\right )Y_i(z), \\ \nonumber
 j=k+1, k+2,\dots \quad k = 0,1,\dots .
\end{eqnarray}
Consider the latter sum in the above equation. Using the matrix
norm properties and applying the H\"older inequality, we find
\begin{eqnarray*}
\Vert
\sum_{i=k}^{j-1}\left(P_j(z_0)Q_i^{+}(z_0)-Q_j(Z_0)P_i^{+}(z_0)
\right )Y_i(z) \Vert \le \\
\le \Vert P_j(z_0) \Vert \sum_{i=k}^{j-1}
  \Vert Q_i^{+}(z_0) \Vert \Vert Y_i(z) \Vert+ \Vert Q_j(z_0) \Vert
 \sum_{i=k}^{j-1} \Vert P_j^{+}(z_0) \Vert \Vert Y_i(z) \Vert \le
 \\ \le ( \Vert P_j(z_0) \Vert + \Vert Q_j(z_0) \Vert
 )M_k^{q,+}(z_0)\left (\sum_{i=k}^{j-1} \Vert Y_i(z) \Vert^p\right
 )^{1/p},
\end{eqnarray*}
where $M_k^{q,+}(z_0)$ is defined by (\ref{mq}). Set $N_{k,j}^p
=\left (\sum_{i=k}^{j-1} \Vert Y_i(z) \Vert^p\right)^{1/p}$ then for
$Y_j(z)$ we get
\begin{equation*}
\Vert Y_j(z) \Vert \le (C_k+\vert z-z_0\vert
M_k^{q,+}(z_0)N_{k,j}^p)( \Vert P_j(z_0) \Vert + \Vert Q_j(z_0)
\Vert), \; j=k+1,k+2,\dots.
\end{equation*}
where $C_k = \max{C_k^1,C_k^2}$. Obviously, the above inequality
also holds when $j=k$. Therefore, for $i=k, k+1,\dots, j-1$ we have:
\begin{equation*}
\Vert Y_i(z) \Vert \le (C_k+\vert z-z_0\vert
M_k^{q,+}(z_0)N_{k,j}^p)( \Vert P_i(z_0) \Vert + \Vert
Q_i(z_0)\Vert).
\end{equation*}
 Now we raise both sides of these inequalities to the $p$-th
power and perform the summation from $i=k$ to $j-1$. Then,
extracting the $p$-th root from the both sides and applying the
Minkowski inequality, we finally get
\begin{equation*}
N_{k,j}^p \le 2C_k M_k^p(z_0)+2\vert z-z_0 \vert M_k^{q,+}(z_0)
M_k^{p}(z_0)N_{k,j}^p.
\end{equation*}
From the assumption of the theorem it follows that both $
M_k^{p}(z_0)$ and $M_k^{q,+}(z_0)$ tend to zero as $k \to \infty$.
Thus we can find an index $k_0$ such that for $k \ge k_0$
\begin{equation*}
\vert z-z_0 \vert M_k^{q,+}(z_0) M_k^{p}(z_0) \le \frac{1}{4}.
\end{equation*}
and therefore
\begin{equation*}
N_{k,j}^p \le 4 C_k M_k^p(z_0)
\end{equation*}
The right-hand side of the above inequality is independent of $j$,
and therefore, $N_{k,j}^p$ has a limit as $j \to \infty$. Since
$\{Y_i(z)\}$ is an arbitrary sequence, it implies that all the
solutions of (\ref{firv}) belong to the space $l^p_n$. For the
equation (\ref{sec}) written in the form
\begin{equation*}
l^{+}(Y^{+}(z))_{j}-z_{0} Y_j^{+}(z)=(z-z_{0})Y_j^{+}(z), \quad j
\ge 0,
\end{equation*}
we obtain the following representation for $Y_j^{+}(z)$ by
substituting (\ref{cp1})-(\ref{cp2}) into (\ref{u2}):
\begin{eqnarray}
\label{rep2} \nonumber
Y_j^{+}(z)= C_k^{1,+}Q_j^{+}(z_0)+C_k^{2,+}P_j^{+}(z_0)+ \\
+(z-z_0)\sum_{i=k}^{j-1}Y_i^{+}(z)\left(Q_i(z_0)P_j^{+}(z_0)-P_i(Z_0)Q_j^{+}(z_0)
\right ), \\ \nonumber
 j=k+1, k+2,\dots \quad k = 0,1,\dots .
\end{eqnarray}
By applying to this formula the same arguments as to (\ref{rep1}),
we find that all the solutions of (\ref{secv}) belong to the space
$l^q_n$.

The case $p=1 (q=\infty) $ is considered separately on the basis of
similar arguments applied to (\ref{rep1}) and (\ref{rep2}). Here
instead of $M_k^{q,+}(z_0)$ one can take
\begin{equation*}
M^{+}(z_0)=\max{ \{\sup_j \Vert P_j^{+}(z_0) \Vert,\sup_j \Vert
Q_j^{+}(z_0) \Vert\}},
\end{equation*}
and $M_k^p(z_0)=M_k^1(z_0)$ is same as above.
\end{proof}

A closer examination of the above proof allows one to establish the
following generalization of the result obtained.
\begin{theorem}
If all solutions of the equations
\begin{equation*}
l(u)_j=0_n, \quad l^{+}(v)_j=0_n^{*}, j \ge 0,
\end{equation*}
where the zero vector column $0_n \in \Bbb C^n$ belong to the spaces
$ l_n^p$ and $ l_n^q$ respectively, where $1/p+1/q=1$, then this is
also true for the solutions of perturbed equations
\begin{equation*}
l(u)_j=F_j u_j, \quad l^{+}(v)_j=v_j^{*}G_j, \quad j \ge 0,
\end{equation*}
where $F_j$ and $G_j \in \Bbb C^{n \times n},$ and the conditions
\begin{equation*}
\sup_{j\ge 0} \Vert F_j \Vert <\infty; \quad \sup_{j\ge 0} \Vert G_j
\Vert <\infty
\end{equation*}
are held.
\end{theorem}

Note that because of the embedding $ l_n^{p_1} \supset l_n^{p_2} $
for $p_1<p_2$ we have that if $ 1 \le p \le 2 $ and the condition of
the Theorem 1 is fulfilled, then all the solutions of (\ref{firv})
and (\ref{secv}) are in $l^p$ for any $ z \in \Bbb C$.

Now consider the matrix $A$ in the symmetric case
\begin{equation}
\label{simm}
 A_{j,j}=A_{j,j}^{*},\quad A_{j+1,j}=A_{j,j+1}>0, \quad
\text{(here * denotes Hermitian conjugation)}
\end{equation}
so $A$ is a Jacobi matrix. In this case we have $ P^{+}(z) =
P^{*}(z)$ and $ Q^{+} (z)= Q^{*}(z) $ for $z \in \Bbb R$ (and the
equation (\ref{secv}) is a conjugate to (\ref{firv})). In view of
the above, we get the following result:
\begin{theorem}
If all solutions of the equation (\ref{firv}) with matrix
coefficients satisfying (\ref{simm}) for some $z=z_0\in \mathbb{R}$
belong to the space $l_n^p, \, 1\le p \le 2$, then this is also true
for any $ z \in \mathbb{C}$.
\end{theorem}
For the case $p=2$ this theorem was proved in ~\cite{K-M} (Theorem
1). Now consider the case $p>2$. As an example, take the following
matrix $A$:
\begin{equation*}
 A_{j,j}= O,\quad A_{j+1,j}=A_{j,j+1}=(j+1)E, \quad
j \ge 0,
\end{equation*}
where $O$ and $E$ are zero and unit matrices of the second order.
Then for the corresponding equation (\ref{firv}) where $z=0$, all
its solutions are in $l_2^{2+\epsilon}$ for any $\epsilon >0$ (one
easily find  by direct calculation of $ P_n(0)$ and $Q_n(0)$ that
both $ \Vert P_n(0) \Vert $ and $\Vert Q_n(0) \Vert \sim n^{-1/2}$
as $ n \to \infty$, and therefore the condition (\ref{mp}) holds
in this case) . However if we take $z=i$ (the imaginary unit) or
$z=-i$, then there exist the solutions of (\ref{firv}) which
belong to $l_2^{\infty},$ but not tend to zero as $j \to \infty$.
Note that for a scalar Jacobi matrix case a similar result was
obtained in ~\cite{Os1} by using the grouping in block approach
offered in ~\cite{N-J}. Here we can apply similar arguments. Thus
we are coming to the following conclusion:
\begin{theorem}
The  Theorem 3 is not valid for $ p >2$.
\end{theorem}

 \end{document}